\documentclass[12pt,a4paper]{article}
\usepackage[english]{babel}
\usepackage{amsmath}
\usepackage{amssymb,amsbsy}
\usepackage{amsthm}
\usepackage[utf8]{inputenc}
\usepackage[T1]{fontenc}
\usepackage{enumerate}
\usepackage{mathrsfs}
\usepackage{bbm}
\usepackage{bm}
\usepackage{color}
\usepackage{hyperref}
\usepackage{authblk}

\usepackage{nicefrac}

\newtheorem{tw}{Theorem}[section]
\newtheorem{lm}[tw]{Lemma}

\newtheorem{pr}[tw]{Proposition}

\theoremstyle{definition}
\newtheorem{df}{Definition}[section]

\newcommand{\R}{\mathbb{R}}
\newcommand{\Z}{\mathbb{Z}}
\newcommand{\N}{\mathbb{N}}
\newcommand{\T}{\mathbb{T}}

\newcommand{\cB}{\mathcal{B}}

\newcommand{\cT}{\mathcal{T}}
\newcommand{\cR}{\mathcal{R}}

\newcommand{\x}{S(x,s)}
\newcommand{\y}{S(x+\delta_{k_j},s)}

\newcommand{\1}{\mathbbm{1}}

\newcommand{\bea}{\begin{eqnarray}}
  \newcommand{\eea}{\end{eqnarray}}
  \newcommand{\beab}{\begin{eqnarray*}}
  \newcommand{\eeab}{\end{eqnarray*}}

  \newcommand{\be}{\begin{equation}}
  \newcommand{\ee}{\end{equation}}
  
\providecommand{\noopsort}[1]{} 

\title{On isomorphism problem for von Neumann flows with one discontinuity}
\author{Adam Kanigowski, Anton Solomko}

\begin{document}
\maketitle

\begin{abstract}
We prove that the absolute value of the slope is a (measure theoretic) invariant in the class of von Neumann special flows with one discontinuity, i.e.\ two ergodic von Neumann flows with one discontinuity are not isomorphic if the slopes of the roof functions have different absolute values, regardless of the irrational rotation in the base.
\end{abstract}

\section{Introduction}
One of the central problems in ergodic theory is the problem of isomorphism of measure preserving systems.
The natural way of dealing with this problem is to introduce (measure theoretic) {\em invariants}. The most classical invariant is given by the {\em entropy}.
It allows to tackle the isomorphism problem between two systems displaying exponential orbit growth.
The central example here are Bernoulli shifts for which, by Ornstein theory \cite{Orn}, the entropy is a complete invariant.
Notice however that entropy gives no information on isomorphism between systems of orbit growth slower than exponential.
For such systems one can study {\em slow entropy}, \cite{KT97}.
This invariant measures the orbit growth in a {\em scale} which can be chosen depending on the class of systems one deals with (e.g. polynomial, logarithmic, etc.).
Other natural invariants are various mixing, spectral or rank properties.

In this paper we consider the isomorphism problem in the class of von Neumann flows \cite{vN32}, i.e.\ special flows built over the rotation $R_\alpha$ of the circle by an irrational number $\alpha$ and under a piecewise $C^1$-function $f$ with non-zero sum of jumps.
Von Neumann flows are weakly mixing \cite{vN32}, never mixing \cite{Ko72}, have purely singular spectrum \cite{FL04}.
Moreover, if the base rotation is of bounded type, they are mildly mixing \cite{FL06}.
Recently it was shown in \cite{KaSo} that the rank of von Neumann flows with roof functions having one discontinuity is infinite.
Notice also that the orbit growth of von Neumann flows is linear (since the sum of jumps is nonzero).
Therefore it seems that slow entropy (in any scale) in the class of von Neumann flows is always constant and hence it cannot be used for the isomorphism problem.
In view of the above properties (mixing, rank, spectrum and slow entropy) it remained an open question what is the isomorphic structure of the class of von Neumann flows.

In the paper we prove that, when the roof function has one discontinuity, the absolute value of the slope is an isomorphism invariant for von Neumann flows.
As a corollary, we get an uncountable family of pairwise non-isomorphic von Neumann flows.
Recall, that two probability preserving flows $\cT=(T_t)_{t\in\R}$ on $(X,\mu)$ and $\cR=(R_t)_{t\in\R}$ on $(Y,\nu)$ are \emph{isomorphic} if there exists an invertible measure preserving transformation $S\colon X\to Y$ such that $S\circ T_t = R_t\circ S$ for every $t\in\R$.
Our main result is the following.

Let $\cT=(T^f_t)_{t\in\R}$ and $\cR=(R^g_t)_{t\in\R}$ be two special flows over irrational rotations $T,R\colon \T\to\T$, $T(x) = x+\alpha$, $R(x)=x+\beta$, under roof functions $f,g\colon\T\to\R_+$, $\int_\T f(x)dx = \int_\T g(x)dx = 1$, of the form
$$
f(x)=A_f\{x\}+f_{ac}(x) \quad \text{ and } \quad g(x)=A_g\{x\}+g_{ac}(x)
$$
respectively, where $\{x\}$ stands for the fractional part of $x$, $A_f,A_g\neq 0$ and $f_{ac},g_{ac}\in C^1(\T)$.

\begin{tw}\label{main}
If $|A_f|\neq|A_g|$, then the flows $\cT=(T^f_t)_{t\in\R}$ and $\cR=(R^g_t)_{t\in\R}$ are not isomorphic.
\end{tw}
\paragraph{Idea of the proof.} As it was mentioned in the introduction, no known slow entropy type invariants can be used for Theorem \ref{main}. For von Neumann flows with one discontinuity the divergence of orbits is caused by two effects:
\emph{slow divergence} by hitting the roof and
\emph{fast divergence} by hitting the discontinuity.
The idea is to take two close points $x,y$ for $\cT$ such that they hit the discontinuity at time $t_0=t_0(x,y)$ before they have time to split by hitting the roof.
This means that before time $t_0$ they are close and at $t_0$ they split by $A_f$ (modulo some $\epsilon$ error) in the flow direction.
So if $\cT$ and $\cR$ were isomorphic, the images $S(x),S(y)$ of $x,y$ would have to be close for most of the time $[0,t_0]$ and split at $t_0$ by $A_f$ in the $\cR$-flow direction.
But points for $\cR$ either split by slow divergence and then (since this divergence is uniform) they have to split by $A_f/2$ at time $t_0/2$ which is a contradiction,
or they split by $A_g$ by hitting the discontinuity. But $|A_g|\neq |A_f|$ which again yields a contradiction.
We use some standard ergodic theory tools, i.e.\ Egorov's theorem, Luzin's theorem, to make this idea precise.

\section{Basic definitions}

We denote by $\T$ the circle group $\R/\Z$ which we will identify with the unit interval $[0,1)$.
For a real number $x$ denote by $\{x\}$ its fractional part, by $[x]=x-\{x\}$ its integer part and by $\|x\|$ its distance to the nearest integer.
Given $x,y\in\T$ with $\|x-y\|<\frac{1}{2}$, $[x,y]$ will denote the shortest interval in $\T$ connecting $x$ and $y$.
Lebesgue measure on $\T$ will be denoted by $\mu$ and Lebesgue measure on $\R$ will be denoted by $\lambda$.
For a bounded function $f\colon\T\to\R_+$, let $f_{\min}:=\inf_{\T}f$ and $f_{\max}:=\sup_{\T}f$.



\subsection{Continued fractions}

For an irrational $\alpha\in\T$, let $[0;a_1,a_2,\ldots]$ stand for the \emph{continued fraction expansion} of $\alpha$ and let $(q_n)_{n\in\N}$ denote the sequence of the \emph{denominators} of $\alpha$ (see \cite{Kh}).
The sequence $(q_n)_{n\in\N}$ is given by the recursive formulas
\begin{equation}\label{cf0}
q_0 = 1, \quad q_1 = a_1, \quad q_{n} = a_{n}q_{n-1} + q_{n-2},
\end{equation}
and the inequality
\begin{equation}\label{cf1}
\frac{1}{2 q_{n+1}} < \| q_n\alpha \| < \frac{1}{q_{n+1}}
\end{equation}
holds for every $n\in\N$.
It follows easily from the above inequality that
\begin{equation}\label{cf2}
\min_{0 \leq i < j < q_n} \|i\alpha - j\alpha\| = \|q_{n-1}\alpha\| > \frac{1}{2q_n}.
\end{equation}
More precisely, the partition of $\T$ by $0,\alpha,...,(q_n-1)\alpha$ has the form
$\{ I_n+k\alpha \,:\, 0\leq k < q_n - q_{n-1} \} \cup \{ I'_n+k\alpha \,:\, 0\leq k < q_{n-1} \}$,
where $I_n = [0,q_{n-1}\alpha]$ and $I'_n = [(-q_{n-1}+q_n)\alpha,0]\subset\T$.

We say that $\alpha$ has \emph{bounded type} if the sequence $(a_n)_{n\in\N}$ is bounded.
Otherwise $\alpha$ is said to have \emph{unbounded type}.


\subsection{Special flows}

Let $T$ be an ergodic automorphism of a standard Borel space $(X,\cB,\mu)$ (with $\mu(X)<+\infty$).
A measurable function $f\colon X\to \R$ defines a cocycle $\Z\times X \to \R$ given by
$$
f^{(n)}(x) =
\begin{cases}
f(x)+f(Tx)+\cdots+f(T^{n-1}x) &\mbox{if}\quad n>0, \\
\hfil 0 &\mbox{if}\quad n=0, \\
\hfil -(f(T^{n}x)+\cdots+f(T^{-1}x)) &\mbox{if}\quad n<0.
\end{cases}
$$
Assume that $f\in L^1(X,\cB,\mu)$ is a strictly positive function.

The \emph{special flow} $\cT = (T^f_t)_{t\in\R}$ over the \emph{base automorphism} $T$ under the \emph{roof function} $f$ is the flow acting on $(X^f,\cB^f,\mu^f)$, where $X^f = \{ (x,s) \in X\times\R \mid 0 \leq s < f(x) \}$ and $\cB^f$ and $\mu^f$ are the restrictions of $\cB \otimes \cB_\R$ and $\mu \otimes \lambda$ to $X^f$ respectively ($\lambda$ stands for Lebesgue measure on $\R$).
Under the action of the flow $\cT$ each point in $X^f$ moves vertically upward with unit speed, and we identify the point $(x,f(x))$ with $(Tx,0)$.
More precisely, for $(x,s)\in X^f$ we have
\begin{equation}\label{spflow}
T^f_t(x,s) = (T^n x, s+t-f^{(n)}(x)),
\end{equation}
where $n\in\Z$ is the unique number such that $f^{(n)}(x) \leq s+t < f^{(n+1)}(x)$.

\subsection{Von Neumann flows}

We call a function $f\colon \T \to \R$ \emph{piecewise $C^1$} if there exist $\beta_1,\ldots,\beta_k\in\T$ such that $f|_{\T\setminus \{\beta_1,\ldots,\beta_k\}}$ is $C^1$ smooth and $f_\pm(\beta_i)=\lim_{x\to \beta_i\pm} f(x)$ is finite for every $i$.
Denote by $d_i:=f_-(\beta_i)-f_+(\beta_i)$ the \emph{jump} of $f$ at point $\beta_i$.
The number $\sum_{i=1}^k d_i$ is the \emph{sum of jumps} of $f$.

\begin{df}
A \emph{von Neumann flow} is a special flow $\cT$ over a rotation $T\colon(\T,\mu)\to(\T,\mu)$ by an irrational $\alpha\in\T$ and under a piecewise $C^1$ roof function $f\colon \T\to\R_+$ with a non-zero sum of jumps.
\end{df}

We will consider the simplest case when $f$ has only one discontinuity.
Without loss of generality we may assume that $f$ is $C^1$ on $\T\setminus \{0\}$ with a jump $-A_f=f_-(0)-f_+(0)\neq 0$.
Any such $f$ can be written in the form
$$
f(x)=A_f\{x\}+f_{ac}(x),
$$
where $f_{ac}\in C^1(\T)$ and $A_f\neq 0$ are such that $f>0$.
In this case we call $A_f$ the \emph{slope} of $f$.
We assume that $\int_\T f d\mu=1$, that is we normalize the resulting measure $\mu^f$ to make it a probability measure.

\medskip
Define the following semimetric on $\T^f$:
\begin{align*}
d^f((x,s),(y,r)) := \min\{
&\|x-y\|+ |s-r|, \\
&\|T(x)-y\|+ |s-r-f(x)|, \\
&\|T^{-1}(x)-y\|+ |s-r+f(T^{-1}(x))|, \\
&\|x-T(y)\|+|s-r+f(y)|, \\
&\|x-T^{-1}(y)\|+|s-r-f(T^{-1}(y))|
\}.
\end{align*}
Informally speaking, two points $(x,s),(y,r)\in\T^f$ are $d^f$-close if their forward or backward images $T^f_t(x,s),T^f_t(y,r)$ are close in the product metric for small $t\in\R$.
The reason we consider $d^f$ instead of the standard product metric is that it makes continuous the mapping $\R\ni t\mapsto T^f_t(x,s)\in \T^f$ for every $(x,s) \in \T$: $d^f((x,s),T^f_t(x,s)) = |t|$ for $|t|<f_{\min}$.
The triangle inequality holds for $d^f$ locally: $d^f((x,s),(y,r))\leq d^f((x,s),(y,r))+d^f((y,r),(z,t))$ whenever $d^f((x,s),(y,r))+d^f((y,r),(z,t))<f_{\min}$. 
Notice also that on a set of measure arbitrary close to 1 the semimetric $d^f$ is equivalent to the product metric.



\section{Main propositions}
In this section we prove several technical propositions which will be used in the proof of Theorem~\ref{main}.
The reader can skip it and jump to the proof of Theorem~\ref{main}, referring back to this section later when needed.

Throughout this section $\cT=(T^f_t)_{t\in\R}$ is a von Neumann special flow over an irrational rotation $R_\alpha\colon\T\to\T$, $R_\alpha(x)=x+\alpha$, under a roof function $f(x)=A_f\{x\}+f_{ac}(x)$, $f_{ac}\in C^1(\T)$.
The following simple lemma will allow us to neglect the $C^1$ part of the roof function in the estimates below.

\begin{lm}[\cite{KaSo}]\label{c1}
Let $g\in C^1(\T)$.
For every $\varepsilon>0$ there exists $\delta>0$ such that for every $x,y\in\T$ with $\|x-y\|<\delta$ and every $n\in\Z$
$$
|g^{(n)}(x)-g^{(n)}(y)|<\varepsilon \max\{1,|n|\|x-y\|\}.
$$
\end{lm}


Let $(q_n)_{n\in\N}$ stand for the sequence of the denominators of $\alpha$.
We will consider two different cases depending on the type of $\alpha$.

\subsection{$\alpha$ of unbounded type}


\begin{pr}\label{pr1}
Assume that $\alpha$ has unbounded type.
For every $\eta>0$ there exist sequences $\delta_n\to 0$ and $E_n\subset \T^f$ such that $\mu^f(\bigcup_{n}E_{k_n})=1$ for every increasing sequence $(k_n)$, $(x+\delta_n,s)\in \T^f$ whenever $(x,s)\in E_n$ and for every $n\in\N$ and $(x,s)\in E_n$ there exist $M_n > n$ and $N_n < -n$ such that the following holds:
\begin{enumerate}[(A)]
\item\label{A} for every $t\in[N_n,M_n]$ we have
$$d^f(T^f_t(x,s),T^f_t(x+\delta_n,s))<\eta;$$
\item\label{B} for every $t\in[M_n,(1+\frac{f_{\min}}{f_{\max}})M_n]$ we have%
\begin{align*}
&d^f(T^f_{t+A_f}(x,s),T^f_{t}(x+\delta_n,s))<\eta \quad \text{if} \quad A_f>0,\\
&d^f(T^f_{t}(x,s),T^f_{t-A_f}(x+\delta_n,s))<\eta \quad \text{if} \quad A_f<0,
\end{align*}
and for every $t\in[(1+\frac{f_{\min}}{f_{\max}})N_n,N_n]$ we have
\begin{align*}
&d^f(T^f_{t-A_f}(x,s),T^f_{t}(x+\delta_n,s))<\eta \quad \text{if} \quad A_f>0,\\
&d^f(T^f_{t}(x,s),T^f_{t+A_f}(x+\delta_n,s))<\eta \quad \text{if} \quad A_f<0.
\end{align*}
\end{enumerate}
\end{pr}
\noindent
The reader should have the following picture in mind.
Both forward and backward images of two `good' points under the flow $\cT$ stay close until they are split by the discontinuity, and once they are split they stay at distance $A_f$ (in the flow direction) for some time.
The fact that $\alpha$ has unbounded type is used in \eqref{B} to guarantee that when the  the interval $[x,x+\delta_n]$ ($\delta_n$ of order $q_n^{-1}$) hits $0$ after $0\leq i_n < q_n$ iterates, then it will also hit $0$ in time $i_n-q_n<0$.

\begin{proof}
Recall that $f=A_f\{\cdot\}+f_{ac}$, where $f_{ac}\in C^1(\T)$.
Let us assume without loss of generality that $A_f>0$, the proof in the other case is analogous.
Fix $\eta>0$, $\eta< \min\{ \frac{1}{50}, A_f^{-2}, f_{\min}\}$, $4\eta^{3/2}<f_{\min}$.
Let $\delta>0$ be such that for every $x,y\in \T$, $\|x-y\|< \delta$, and every $k\in \Z$
\begin{equation}\label{c1part}
|f_{ac}^{(k)}(x)-f_{ac}^{(k)}(y)|<\eta^{3/2}\max\{ 1, |k| \|x-y\| \}
\end{equation}
(see Lemma~\ref{c1}).
By the assumption on $\alpha$, there exists a sequence $(l_n)$ such that $\frac{q_{l_n+1}}{q_{l_n}}\to +\infty$.
We will assume (by replacing $(l_n)$ with its subsequence if necessary) that $q_{l_1} > \frac{\eta^{2}}{\delta}$, $l_n>nf_{\min}^{-1}$, $q_{l_n}>4l_n$ and $\frac{q_{l_n+1}}{q_{l_n}}> \frac{4}{\eta^{3}}$.
Set $\delta_n:=\frac{\eta^2}{q_{l_n}}$, $\delta_n < \delta$ for all $n$.
Let $I_n:=[-\frac{\eta^3}{q_{l_n}},-\frac{2}{q_{l_n+1}}]\subset\T$.
Notice that by \eqref{cf1} and \eqref{cf2}, for every $x\in I_n$,
\begin{equation}\label{visits}
\{i\in\{-q_{l_n},\ldots,2q_{l_n}\} \,:\, i\alpha\in(x,x+\delta_n) \} = \{-q_{l_n},0,q_{l_n},2q_{l_n}\}.
\end{equation}
Set $B_n:=\bigsqcup_{i=-l_n}^{-q_{l_n}+l_n}R^i_\alpha I_n\subset\T$ and let
\begin{equation}\label{aen}
E_n:=\{(x,s)\in \T^f\,:\, x\in B_n \text { and } (x+\delta_n,s)\in \T^f\} \subset\T^f.
\end{equation}
We show in Lemma \ref{indep} below that $\mu^f(\bigcup_n E_{k_n})=1$ for every increasing sequence $(k_n)$.

Take any $(x,s)\in E_n$.
Let $i_n\in [l_n, q_{l_n}-l_n]$ be the unique integer such that $0\in R_\alpha^{i_n} [x, x+\delta_n]$, $0\in R_\alpha^{i_n-q_{l_n}} [x, x+\delta_n]$ and
\begin{equation}\label{nojump}
0\notin R_\alpha^{k} [x, x+\delta_n]
\end{equation}
for all $i_n-2q_{l_n}< k<i_n+q_{l_n}$, $k\neq i_n, i_n-q_{l_n}$.
We have then for every $k\in(i_n-q_{l_n},i_n]$, 
\begin{multline}\label{est}
|f^{(k)}(x)-f^{(k)}(x+\delta_n)|
\leq |f_{ac}^{(k)}(x)-f_{ac}^{(k)}(x+\delta_n)| + A_f|k|\delta_n\\
\leq \eta^{3/2}\max\{1,|k|\delta_n\} + A_f|k|\delta_n
= \eta^{3/2} + A_f|k|\delta_n
< 2\eta^{3/2}.
\end{multline}
Let $M_n:=f^{(i_n+1)}(x+\delta_n)-s$ and $N_n:=f^{(i_n-q_{l_n})}(x+\delta_n)-s$.
Then for every $t\in [N_n,M_n]$,
$$
d^f(T_t^f(x,s),T_t^f(x+\delta_n,s)) < \delta_n +2\eta^{3/2}<\eta.
$$
Since $l_n\leq i_n\leq q_{l_n}-l_n$, we have $M_n\geq l_nf_{\min}>n$, $N_n \leq -l_nf_{\min}<-n$.
This proves \eqref{A}.
Let us prove \eqref{B} for $t\in [M_n,(1+\frac{f_{\min}}{f_{\max}})M_n]$,
the proof for $t\in[(1+\frac{f_{\min}}{f_{\max}})N_n,N_n]$ is analogous.
By the cocycle equality we can write for every $k\in(i_n,i_n+q_{l_n})$ and $y=x,x+\delta_n$:
$$
f^{(k)}(y) = f^{(i_n)}(y) + f(R_\alpha^{i_n}(y)) + f^{(k-i_n-1)}(R_\alpha^{i_n+1}(y)).
$$
Since $0\in R_\alpha^{i_n} [x, x+\delta_n]$,
\begin{equation}\label{est2}
|f(R_\alpha^{i_n}(x))-f(R_\alpha^{i_n}(x+\delta_n))-A_f| < \eta^{3/2} + A_f\delta_n < 2\eta^{3/2}.
\end{equation}
By \eqref{nojump} we also know that for every $k\in(i_n,i_n+q_{l_n})$
\begin{equation}\label{est3}
|f^{(k-i_n-1)}(R_\alpha^{i_n+1}(x)) - f^{(k-i_n-1)}(R_\alpha^{i_n+1}(x+\delta_n))| < 2\eta^{3/2}.
\end{equation}
Combining \eqref{est}, \eqref{est2} and \eqref{est3}, we get
$$
|f^{(k)}(x)-f^{(k)}(x+\delta_n)-A_f|< 6\eta^{3/2}
$$
for every $k\in (i_n, i_n+q_{l_n})$.
Therefore for every $t\in [M_n, f^{(i_n+q_{l_n}+1)}(x+\delta_n)-s]$ we have
$$
d^f(T^f_{t+A_f}(x,s),T_{t}^f(x+\delta_n,s)) < \delta_n + 6\eta^{3/2} < \eta.
$$
Since $M_n\leq q_{l_n}f_{\max}$, while $f^{(i_n+q_{l_n}+1)}(x+\delta_n)-s - M_n = f^{(q_n)}(R_\alpha^{i_n+1}(x+\delta_n))\geq q_{l_n}f_{\min}$,
this finishes the proof of \eqref{B}.
\end{proof}

To prove Lemma~\ref{indep} we will use the following strong law of large numbers for weakly correlated random variables:
\begin{lm}[\cite{Ly88}] \label{wLLN}
Let $(X,\mu)$ be a probability space.
Given a sequence of measurable subsets $B_n\subset X$, if
\begin{equation} \label{wc}
|\mu(B_n\cap B_m) - \mu(B_n) \mu(B_m)| \leqslant \varepsilon(|n-m|),
\end{equation}
where $\sum_{k\in\N} \frac{\varepsilon(k)}{k} < \infty$, then
$\frac{1}{n} \sum_{i=1}^n(\1_{B_i} - \mu(B_i))\to 0$ $\mu$-almost everywhere.
In particular, if\, $\inf_n\mu(B_n)>0$, then $\mu(\limsup_n B_n)=1$.
\end{lm}

\begin{lm}\label{indep}
Let $(E_n)$ be as in \eqref{aen}.
Then for every increasing sequence $(k_n)$ we have
$$
\mu^f(\bigcup_n E_{k_n})=1.
$$
\end{lm}
\begin{proof}
Recall that $B_n=\bigsqcup_{i=-l_n}^{-q_{l_n}+l_n}R^i_\alpha I_n$, where $I_n=[-\frac{\eta^3}{q_{l_n}},-\frac{2}{q_{l_n+1}}]\subset\T$, and $E_n=\{(x,s)\in \T^f\,:\, x\in B_n \text { and } (x+\delta_n,s)\in \T^f\} \subset\T^f$.
The reader can easily check that the sets $R^i_\alpha I_n$ are pairwise disjoint.
Let $B^f_n := \{ (x,s)\in\T^f \mid x\in B_n \}$ and $X_n := \{ (x,s)\in\T^f \mid (x+\delta_n,s)\in\T^f \}$, so that $E_n = B^f_n \cap X_n$.
Obviously, $X_1\subset X_2\subset\cdots$ and $\mu^f(X_n)\to 1$, so it is enough to prove that $\mu(\bigcup_n B_{k_n})=1$ for every increasing sequence $(k_n)$.
For this we will apply Lemma~\ref{wLLN}.

Denote $p(n) := q_{l_n}-2l_n+1$ and $s_{m}(n):=|\{i\in[-q_{l_n}+l_n,-l_n] \::\: i\alpha \in B_m\}|$, $m<n$.
Let $(k_n)_{n\in\N}$ be an increasing sequence.
By passing to a further subsequence if necessary we may assume without loss of generality that $$\left|\frac{s_{k_m}(k_n)}{p(k_n)}-\mu(B_{k_m})\right|<\frac{1}{2^n} \quad \text{and} \quad \frac{p(k_m)}{p(k_n)}<\frac{1}{2^n}$$ for every $m<n$
(the first inequality holds for every $k_m$ ($m<n$) fixed and $k_n$ large enough by unique ergodicity of $R_\alpha$).
It is a routine to verify that
$$
\frac{s_{k_m}(k_n)-2p(k_m)}{p(k_n)}\mu(B_{k_n})
\leq \mu(B_{k_n}\cap B_{k_m})
\leq \frac{s_{k_m}(k_n)+2p(k_m)}{p(k_n)}\mu(B_{k_n})
$$
and hence
$$|\mu(B_{k_n}\cap B_{k_m}) - \mu(B_{k_n})\mu(B_{k_m})|\leq \left|\frac{s_{k_m}(k_n)}{p(k_n)}-\mu(B_{k_m})\right|+2 \frac{p(k_m)}{p(k_n)}<\frac{3}{2^n}.$$
Therefore, \eqref{wc} holds for the sequence $(B_{k_n})_{n\in\N}$ with $\varepsilon(n)=\frac{3}{2^n}$.
Finally, observe that $\mu(B_n)=p(n) \mu(I_n) =(q_{l_n}-2l_n+1)(\frac{\eta^3}{q_{l_n}}-\frac{2}{q_{l_n+1}}) > \frac{q_{l_n}}{2} \cdot \frac{\eta^3}{2q_{l_n}} =  \frac{\eta^3}{4}$ for every $n$ by the assumptions on $(l_n)$.
It follows from Lemma~\ref{wLLN} that $\mu(\limsup_n B_{k_n})=1$.
\end{proof}

\begin{pr}\label{pr2}
Let $\alpha$ be any irrational.
For every $\xi>0$ small enough and $M>0$ there exists a set $Z\subset \T^f$, $\mu^f(Z)>1-\xi$, and $\theta_{\xi,M}>0$ such that for every $(x,s),(y,r)\in Z$, $d^f((x,s),(y,r))<\theta_{\xi,M}$, $x\neq y$, there exists $L > M$ such that at least one of the following holds:
\begin{enumerate}[(a)]
\item\label{prop2a}
for every $t\in[0,L]$ we have
$$
d^f(T^f_t(x,s),T^f_t(y,r))<2\xi,
$$
and for every $t\in [L,(1+\frac{f_{\min}}{10f_{\max}})L]$ we have
\begin{align*}
&d^f(T^f_t(x,s),T^f_{t+\xi}(y,r))<\xi/2 \quad \text{if} \quad A_f>0,\\
&d^f(T^f_t(x,s),T^f_{t-\xi}(y,r))<\xi/2 \quad \text{if} \quad A_f<0;
\end{align*}

\item\label{prop2b}
for every $t\in[-L,0]$ we have
$$
d^f(T^f_t(x,s),T^f_t(y,r))<2\xi,
$$
and for every $t\in [-(1+\frac{f_{\min}}{10f_{\max}})L,-L]$ we have
\begin{align*}
&d^f(T^f_t(x,s),T^f_{t-\xi}(y,r))<\xi/2 \quad \text{if} \quad A_f>0,\\
&d^f(T^f_t(x,s),T^f_{t+\xi}(y,r))<\xi/2 \quad \text{if} \quad A_f<0;
\end{align*}

\item\label{prop2c}
for every $t\in[0,L]$ we have
$$
d^f(T^f_t(x,s),T^f_t(y,r))<15\xi,
$$
and for every $t\in [L, (1+\frac{f_{\min}}{f_{\max}})L]$ we have
\begin{align*}
&d^f(T^f_{t+A_f}(x,s),T^f_{t}(y,r))<30\xi \quad \text{if} \quad A_f>0,\\
&d^f(T^f_t(x,s),T^f_{t-A_f}(y,r))<30\xi \quad \text{if} \quad A_f<0.
\end{align*}
\end{enumerate}
\end{pr}

For the proof of Proposition~\ref{pr2} we will need the following lemma: 

\begin{lm}[{\cite[Lemma~3.3]{KaSo}}]\label{lin}
Fix $x,y\in \T$, $x\neq y$, and let $n\in \N$ be any integer such that
\begin{equation}\label{distanc}
\|x-y\|< \frac{1}{6q_n}.
\end{equation}
Then one of the following holds:
\begin{enumerate}[(i)]
\item\label{lem3a} $0\notin \bigcup_{k=0}^{\left[\frac{q_{n+1}}{6}\right]}R^k_\alpha[x,y]$;
\item\label{lem3b} $0\notin \bigcup_{k=-\left[\frac{q_{n+1}}{6}\right]}^{0}R^k_\alpha[x,y]$;
\item\label{lem3c} $0\in \bigcup_{k=0}^{q_n-1}R^k_\alpha[x,y]$.
\end{enumerate}
\end{lm}

\begin{proof}[Proof of Proposition \ref{pr2}]
We give the proof assuming that $A_f>0$, the proof in the other case is analogous.
We assume that $\xi<\min\{\frac{f_{\min}}{4},\frac{A_f}{72}\}$.
Let $\theta_{\xi,M}>0$ be such that for every $x,y\in \T$, $\|x-y\|< \theta_{\xi,M}$, and every $k\in \Z$
\begin{equation}\label{ac2}
|f_{ac}^{(k)}(x)-f_{ac}^{(k)}(y)|<\frac{\xi}{20}\max\{ 1, |k| \|x-y\| \}
\end{equation}
(Lemma~\ref{c1}).
We will also assume that $\theta_{\xi,M} < \min\{\frac{\xi}{10}, \frac{\xi}{20A_f}, 
\frac{12\xi}{A_f q_{n_0}}\}$,
where $n_0$ is such that $q_{n_0}>\max\{{12M}{f_{\min}^{-1}},{M^2}{f_{\min}^{-2}}\}$.
Let
$$
Z_1:=\{(x,s)\in\T^f\,:\,\theta_{\xi,M}<s<f(x)-\theta_{\xi,M}\},
$$
so that for every $(x,s),(y,r)\in Z_1$, $d^f((x,s),(y,r))=\|x-y\|+|s-r|$ if $d^f((x,s),(y,r))<\theta_{\xi,M}$, $\mu^f(Z_1)>1-\xi/2$.
Let $B_n:=\bigcup_{i=-[\sqrt{q_n}]}^0 R_\alpha^i[-\frac{1}{6q_n},\frac{1}{6q_n}]$, $\mu(B_n)< \frac{1}{\sqrt{q_n}}$.
Since $\sum_{n=1}^\infty \frac{1}{\sqrt{q_n}}<\infty$, as follows easily from \eqref{cf0}, we can find $m\in\N$ such that $\mu(\bigcup_{n=m}^\infty B_n)< \frac{\xi}{2f_{\max}}$.
Define
$$
Z_2:=\{(x,s)\in\T^f \,:\, x\notin \bigcup_{n=m}^\infty B_n \},
$$
$\mu^f(Z_2)>1-\xi/2$,
and let $Z:=Z_1\cap Z_2$.

Take any $(x,s),(y,r)\in Z$ with $d^f((x,s),(y,r))<\theta_{\xi,M}$, $x\neq y$.
Let $n\in\N$ be unique such that
$$
\frac{12\xi}{A_f q_{n+1}} \leq \|x-y\|< \frac{12\xi}{A_f q_n}.
$$
Notice that $n\geq n_0$.
Since $n$ satisfies \eqref{distanc} for $x,y$, one of \eqref{lem3a}--\eqref{lem3c} in Lemma~\ref{lin} holds.

Assume first that $x,y$ satisfy \eqref{lem3a} (the forward orbit of $[x,y]$ avoids the discontinuity up to time $\frac{q_{n+1}}{6}$).
We will prove \eqref{prop2a} in this case.
Let $k_0:=\big[ \frac{\xi}{A_f\|x-y\|} \big]$.
By the assumptions on $\xi$ and $\theta_{\xi,M}$ we have ${q_{n}}/{12} \leq k_0 \leq {q_{n+1}}/{12}$.
For every $k\in[0,k_0]$, by \eqref{ac2} and the definition of $k_0$,
$$
|f^{(k)}(x)-f^{(k)}(y)| \leq |f_{ac}^{(k)}(x)-f_{ac}^{(k)}(y)| + A_f|k|\|x-y\| <  {3\xi}/{2}.
$$
Let $L:=f^{(k_0+1)}(x)-s$, $L>f_{\min}q_n/12\geq f_{\min}q_{n_0}/12>M$.
For every $t\in [0,L]$,
$$
d^f(T_t^f(x,s),T_t^f(y,r)) < \theta_{\xi,M} + {3\xi}/{2} < 2\xi.
$$
We also have
\begin{align*}
|f^{(k_0)}(x) - f^{(k_0)}(y) - \xi|
&\leq |f^{(k_0)}_{ac}(x) - f^{(k_0)}_{ac}(y)| + |A_f k_0 \|x-y\| - \xi| \\
&< {\xi}/{20} + A_f \theta_{\xi,M} < {\xi}/{10}.
\end{align*}
Since $0\notin \bigcup_{k=k_0}^{2k_0} R_\alpha^k[x,y]$, we have for every $k\in[k_0,6k_0/5]\cap\Z$,
$$
|f^{(k-k_0)}(x+k_0\alpha)-f^{(k-k_0)}(y+k_0\alpha)| < {\xi}/{20} + A_f|k-k_0|\theta_{\xi,M} < {3\xi}/{10}.
$$
From the above two inequalities and the cocycle equality we get
for every $k\in[k_0,6k_0/5]\cap\Z$
\begin{align*}
|f^{(k)}(x)-f^{(k)}(y)-\xi|
&\leq |f^{(k_0)}(x) - f^{(k_0)}(y) - \xi| \\
&+ |f^{(k-k_0)}(x+k_0\alpha)-f^{(k-k_0)}(y+k_0\alpha)|\\
&< {4\xi}/{10}.
\end{align*}
Therefore, for all $t\in[f^{(k_0)}(x)-s, f^{([6k_0/5])}(x)-s]$,
$$
d^f(T_{t}^f(x,s),T_{t-\xi}^f(y,r)) < \theta_{\xi,M} + {4\xi}/{10} < {\xi}/{2}.
$$
This completes the proof of \eqref{prop2a}.

If \eqref{lem3b} holds for $x,y$, one can similarly prove \eqref{prop2b}.
The proof in this case goes along the same lines, one just needs to switch the time direction
from positive to negative, therefore we skip it.
It remains to consider the case \eqref{lem3c} and prove \eqref{prop2c}.

If $x,y$ satisfy \eqref{lem3c}, there exists $0\leq k_0 < q_n$ such that $0\in R_\alpha^{k_0}[x,y]$.
By the definition of $Z$, $k_0 > \sqrt{q_n}$.
We also know from \eqref{cf2} that $0\notin R_\alpha^{k}[x,y]$ for all $0\leq k < k_0$ and $k_0 < k < k_0 + q_n$.
For every $k\in[0,k_0]$,
$$
|f^{(k)}(x)-f^{(k)}(y)| \leq |f_{ac}^{(k)}(x)-f_{ac}^{(k)}(y)| + A_f k\|x-y\| <  {\xi}/{20} + 12{\xi} < {13\xi}.
$$
Let $L:=f^{(k_0+1)}(y)-r$, $L>f_{\min}\sqrt{q_n}\geq f_{\min}\sqrt{q_{n_0}}>M$.
For every $t\in [0,L]$ we have
$$
d^f(T_t^f(x,s),T_t^f(y,r)) < \theta_{\xi,M} + {13\xi} < 14\xi.
$$
On the other hand by the definition of $k_0$ we have
\begin{multline*}
|f^{(k_0+1)}(x)-f^{(k_0+1)}(y)-A_f| \leq |f_{ac}^{(k_0+1)}(x)-f_{ac}^{(k_0+1)}(y)|+\\
+A_f\theta_{\xi,M}+ A_f (k_0+1)\|x-y\| <  {\xi}/{20} + {\xi}/{20}+ 12\xi < {13\xi}.
\end{multline*}
Since $0\notin \bigcup_{k=k_0}^{2k_0} R_\alpha^k[x,y]$, we have for every $k\in(k_0,2k_0]\cap\Z$,
$$
|f^{(k-k_0)}(x+k_0\alpha)-f^{(k-k_0)}(y+k_0\alpha)| <13\xi.
$$
Finally, for every $t\in [L,(1+\frac{f_{\min}}{f_{\max}})L]$, we get
$$d^f(T^f_{t+A_f}(x,s),T^f_{t}(y,r))<30\xi.$$
This finishes the proof of \eqref{prop2c}.

\end{proof}

\subsection{$\alpha$ of bounded type}
For $\alpha$ of bounded type we have a weaker statement than Proposition \ref{pr1} (see Proposition \ref{pr3} below). Therefore we need a stronger statement than Proposition \ref{pr2} (see Proposition \ref{pr4} below). We will use Propositions \ref{pr3} and \ref{pr4} in the proof of Theorem \ref{main} in the case both $\alpha$ and $\beta$ are of bounded type.

\begin{pr}\label{pr3}
Assume that $\alpha$ has bounded type.
For every $\eta>0$ there exist sequences $\delta_n\to 0$ and $E_n\subset \T^f$ such that $\mu^f(\bigcup_{n}E_{k_n})=1$ for every increasing sequence $(k_n)$, $(x+\delta_n,s)\in \T^f$ whenever $(x,s)\in E_n$ and for every $n\in\N$ and $(x,s)\in E_n$ there exists $M_n>n$ such that the following holds:
\begin{enumerate}
\item[$(A')$]\label{A'} for every $t\in[0,M_n]$ we have
$$d^f(T^f_t(x,s),T^f_t(x+\delta_n,s))<\eta;$$
\item[$(B')$]\label{B'} for every $t\in[M_n,(1+\frac{f_{\min}}{f_{\max}})M_n]$ we have%
\begin{align*}
&d^f(T^f_{t+A_f}(x,s),T^f_{t}(x+\delta_n,s))<\eta \quad \text{if} \quad A_f>0,\\
&d^f(T^f_{t}(x,s),T^f_{t-A_f}(x+\delta_n,s))<\eta \quad \text{if} \quad A_f<0.
\end{align*}
\end{enumerate}
\end{pr}


\begin{proof} The proof, up to one difference, follows the same lines as the proof of Proposition \ref{pr1}.
In this case $(l_n)$ can be any (sparse) subsequence of $\N$, so that $q_{l_1}\geq \frac{\eta^2}{\delta}$, $q_{l_n}\geq 4l_n$. Let $\delta_n:=\frac{\eta^2}{q_{l_n}}$.
We define sets $I_n$, $B_n$ and $E_n$ the same way as in Proposition \ref{pr1} (notice that \eqref{visits} does not hold anymore).
Then it follows by Lemma~\ref{indep} that $\mu^f(\bigcup_nE_{k_n})=1$ for every increasing sequence $(k_n)$.

We define $i_n$ the same way, then $0\in R_\alpha^{i_n}[x,x+\delta_n]$ (contrary to the proof of Proposition \ref{pr1}, it is not true that $0\in R_\alpha^{i_n-q_{l_n}}[x,x+\delta_n]$) and for every $i_n<k<i_n+q_{l_n}$ \eqref{nojump} holds.
Then for every $k\in [0,i_n]$, we have
$$
|f^{(k)}(x)-f^{(k)}(x+\delta_n)|< 2\eta^{3/2}.
$$
We define $M_n:=f^{(i_n+1)}(x+\delta_n)-s$. From this point on, the proof is a word by word repetition of the proof of Proposition \ref{pr1} starting below \eqref{est}.
\end{proof}

\begin{pr}\label{pr4}
Assume that $\alpha$ has bounded type.
For every $\xi>0$ small enough there exists a set $Z\subset \T^f$, $\mu^f(Z)>1-\xi$, and $\theta_{\xi,M}>0$ such that for every $(x,s),(y,r)\in Z$, $d^f((x,s),(y,r))<\theta_{\xi,M}$, $x\neq y$, there exists 
$L > M$ such that at least one of the following holds:
\begin{enumerate}
\item[$(a')$]
for every $t\in[0,L]$ we have
$$
d^f(T^f_t(x,s),T^f_t(y,r))<2\xi,
$$
and for every $t\in [L,(1+\frac{f_{\min}}{10f_{\max}})L]$ we have
\begin{align*}
&d^f(T^f_t(x,s),T^f_{t+\xi}(y,r))<\xi/2 \quad \text{if} \quad A_f>0,\\
&d^f(T^f_t(x,s),T^f_{t-\xi}(y,r))<\xi/2 \quad \text{if} \quad A_f<0;
\end{align*}
\item[$(b')$]
for every $t\in[0,L]$ we have
$$
d^f(T^f_t(x,s),T^f_t(y,r))<15\xi,
$$
and for every $t\in [L, (1+\frac{f_{\min}}{f_{\max}})L]$ we have
\begin{align*}
&d^f(T^f_{t+A_f}(x,s),T^f_{t}(y,r))<30\xi \quad \text{if} \quad A_f>0,\\
&d^f(T^f_t(x,s),T^f_{t-A_f}(y,r))<30\xi \quad \text{if} \quad A_f<0.
\end{align*}
\end{enumerate}

\end{pr}

\begin{proof}
The proof is similar to the proof of Proposition \ref{pr2}.
Assume $A_f>0$.
Fix $\xi< (10\sup_{n}a_n+1)^{-2}$.
We define $\theta_{\xi,M},Z_1,Z_2$ and $Z$ the same way as in the proof of Proposition \ref{pr2}. Take any $(x,s),(y,r)\in Z$ and let $n$ be unique satisfying
$$
\frac{1}{2q_{n+1}} \leq \|x-y\|< \frac{1}{2q_n}.
$$
Denote $k_0:= \big[\frac{\xi}{A_f\|x-y\|}\big]$. We have the following cases:

\textbf{Case~1}: $0\notin \bigcup_{i=0}^{2k_0}R_\alpha^i[x,y]$.
Then we show $(a')$.
Notice that for every $k\in [0,k_0]$ we have
$$
|f^{(k)}(x)-f^{(k)}(y)| \leq |f_{ac}^{(k)}(x)-f_{ac}^{(k)}(y)| + A_f|k|\|x-y\| <  3\xi/2
$$
and therefore for $L:=f^{(k_0)}(x)-s$ we have for $t\in [0,L]$
$$
d^f(T_t^f(x,s),T_t^f(y,r))<2\xi.
$$
Moreover, for every $k\in [k_0,6k_0/5]$, we have
$$
|f^{(k)}(x)-f^{(k)}(y)-\xi| < 4\xi/10.
$$
Therefore, for every $t\in [L, 1+\frac{f_{\min}}{10f_{\max}}L]$, we have
$$
d^f(T_t^f(x,s),T_{t+\xi}^f(y,r))<\xi/2,
$$
which gives $(a')$.

\textbf{Case~2}: $0\in \bigcup_{i=0}^{2k_0}R_\alpha^i[x,y]$.
Then let $i_0$ be such that $0\in [x+i_0\alpha, y+i_0\alpha]$.
Since $x,y\in Z$ it follows that $i_0 > \sqrt{q_n}>M$ by the choice of $\theta_{\xi,M}$.
Let $L:=f^{(i_0)}(x)-s$.
Analogously to the previous case we show that for $t\in[0,L]$ we have
$$
d^f(T_t^f(x,s),T_t^f(y,r))<4\xi.
$$
Moreover since $\alpha$ is of bounded type, it follows that $i_0\leq 2k_0< q_n$, and therefore $0\notin R_\alpha^i[x,y]$ for $i\in [i_0+1,2i_0]$.
Hence for every $t\in [L, 1+\frac{f_{\min}}{f_{\max}}L]$, we have
$$
d^f(T_{t+A_f}^f(x,s),T_{t}^f(y,r))<30\xi.
$$
This finishes the proof.
\end{proof}

\section{Proof of Theorem \ref{main}}
In this section we will use Propositions \ref{pr1}, \ref{pr2}, \ref{pr3} and \ref{pr4} from the previous section to prove Theorem \ref{main}.

\begin{proof}[Proof of Theorem \ref{main}]

We will argue by contradiction.
Let $S:\T^f\to \T^g$ be an isomorphism between $\cT$ and $\cR$, i.e. $S\circ T^f_t=R_t^g\circ S$ for every $t\in\R$.
We will consider separately two cases.

\textbf{Case 1.} \textit{Either $\alpha$ or $\beta$ has unbounded type.}
Assume for definiteness that $\alpha$ has unbounded type.
We will use Proposition \ref{pr1} for $\alpha$ and Proposition \ref{pr2} for $\beta$.
By Luzin's theorem, there exists a set $X\subset \T^f$, $\mu^f(X)> 1-\frac{f_{\min}}{20f_{\max}}$, such that for any $\varepsilon>0$ there is $\eta = \eta(\varepsilon)>0$
such that for every $(x,s),(y,r) \in X$,
\begin{equation}\label{unifcont1}
d^f((x,s),(y,r))<\eta \quad \Rightarrow \quad d^g(S(x,s),S(y,r))<\varepsilon.
\end{equation}
Since $\cT$ is ergodic, for $\mu$-a.e.\ $(x,s)\in \T^f$,
$$
\lim_{M\to\infty}\frac{1}{M} \lambda\{ t\in[0,M] \,:\, T^f_{t+\iota A_f}(x,s) \in X \} = \mu^f(X),
$$
$\iota=-1,0,1$.
Therefore, by Egorov's theorem, there exists a set $Y\subset \T^f$, $\mu^f(Y)> 0.99$, and $M_0>0$ such that for every $(x,s)\in Y$ and every $N<M$, $|M|,|N|\geq M_0$, $\frac{|M-N|}{|M|}\geq \min\{ \frac{f_{\min}}{2f_{\max}}, \frac{g_{\min}}{20g_{\max}} \}$, we have
\begin{equation}\label{yeg1}
\lambda\{t\in [N,M]\;:\; T_{t+ \iota A_f}^f(x,s)\in X, \,\iota=-1,0,1\}\geq 0.9|M-N|.
\end{equation}
Fix $0<\xi\ll f_{\min}, g_{\min}, A_f, A_g, \min\{\|k\beta\|:0<k\leq \frac{|A_f|+|A_g|}{g_{\min}}+1\}$ and let $Z\subset \T^g$ and $\theta_{\xi,M_0}<\xi/10$ come from Proposition~\ref{pr2}, $\mu^g(Z)>0.99$.
Choose $\eta>0$ in \eqref{unifcont1} so that for all $(x,s),(y,r)\in X$
\begin{equation}\label{unifcont2}
d^f((x,s),(y,r))<\eta \quad \Rightarrow \quad d^g(S(x,s),S(y,r))<\theta_{\xi,M_0}.
\end{equation}
Consider the set
\begin{equation}\label{defv}
V:=X\cap Y\cap S^{-1}Z\subset \T^f,
\end{equation}
$\mu^f(V)>0.9$.
Let $(E_n)$ and $(\delta_n)$ be the sequences coming from Proposition~\ref{pr1} for $\eta$ as above.
Since $\delta_n\to 0$, there exists an increasing sequence $(k_n)$ such that the following set
$$
V_0= \{(x,s)\in V \,:\, (x+\delta_{k_n},s) \in V \text{ for every } n\in\N \}
$$
has positive measure. 
We may assume without loss of generality that $\delta_{k_1}<\eta$ and $k_1>M_0$.
By Proposition~\ref{pr1}, $\mu^f(\bigcup_n E_{k_n}\cap V_0)>0$.
Fix any $(x,s)\in \bigcup_n E_{k_n}\cap V_0$.
By the definition of $V_0$, there exists $j\in \N$ such that
$$
(x,s)\in E_{k_j}\text{ and } (x,s),(x+\delta_{k_j},s)\in V.
$$
Let $M_{k_j}, N_{k_j}$ be as in Proposition~\ref{pr1}, so that \eqref{A} and \eqref{B} holds for the point $(x,s)$,
$M_{k_j}>M_0$, $N_{k_j}<-M_0$.
By the definition of $V$ it follows that $\x,\y\in Z$ and $d^g(\x,\y) < \theta_{\xi,M_0}$.
Therefore, by Proposition~\ref{pr2} (for the flow $\cR$ and the roof function $g$), there exists $L>M_0$ such that one of \eqref{prop2a}, \eqref{prop2b} or \eqref{prop2c} holds for $\x, \y$.
We will conduct the proof assuming that \eqref{prop2a} or \eqref{prop2c} holds.
The proof in the case \eqref{prop2b} is completely analogous to the proof in the case \eqref{prop2a}, one just needs to go backward instead of going forward.
Denote
$$
P:=\{t\in \R\;:\; T_{t+ \iota A_f}^f(x,s),T_{t+ \iota A_f}^f(x+\delta_{k_j},s)\in X, \,\iota=-1,0,1\}.
$$
Since $(x,s),(x+\delta_{k_j},s)\in V\subset Y$, it follows by \eqref{yeg1} that at least one of the following sets is nonempty:
\begin{align*}
I_1 &:= P\cap[0,M_{k_j}]\cap[L,(1+\tfrac{g_{\min}}{10g_{\max}})L], \\
I_2 &:= P\cap [M_{k_j}, (1+\tfrac{f_{\min}}{f_{\max}})M_{k_j}] \cap[L,(1+\tfrac{g_{\min}}{10g_{\max}})L], \\
I_3 &:= P\cap [M_{k_j},(1+\tfrac{f_{\min}}{f_{\max}})M_{k_j}]\cap[0,L].
\end{align*}
For every $t\in P\cap[0,M_{k_j}]$ we have by \eqref{A} and \eqref{unifcont2}
\begin{equation}\label{eqA}
d^g(R^g_t(\x), R^g_t(\y))<\theta_{\xi,M_0}<\xi/10,
\end{equation}
and for every $t\in P\cap[M_{k_j},(1+\tfrac{f_{\min}}{f_{\max}})M_{k_j})$ we have by \eqref{B} and \eqref{unifcont2}
\begin{equation}\label{eqB}
d^g(R^g_t(\x), R^g_{t\pm A_f}(\y))<\theta_{\xi,M_0}<\xi/10.
\end{equation}
On the other hand, by Proposition~\ref{pr2}, for every $t\in[0,L]$ we have
\begin{equation}\label{eq0}
d^g(R^g_t(\x), R^g_t(\y))<15\xi
\end{equation}
and for every $t\in[L,(1+\tfrac{g_{\min}}{10g_{\max}})L]$ we have either
\begin{equation}\label{eqa2}
d^g(R^g_t(\x), R^g_{t\pm \xi}(\y))<\xi/2,
\end{equation}
or
\begin{equation}\label{eqc2}
d^g(R^g_t(\x), R^g_{t\pm A_g}(\y))<30\xi.
\end{equation}
Denote $y:=\y$.
Combining \eqref{eqA} and \eqref{eqB} with \eqref{eq0}, \eqref{eqa2} or \eqref{eqc2} we get by the triangle inequality that for $t\in I_1\cup I_2\cup I_3$,
$$
d^g(R^g_t(y),R^g_{t\pm\xi}(y))<\xi \quad \text{or} \quad d^g(R^g_t(y),R^g_{t\pm a}(y))<40\xi
$$
for some $a\in\{A_f,A_g,A_f\pm A_g,A_f\pm\xi\}$.
But none of the above two inequalities can hold, because $d^g( R^g_t(y), (R^g_{t\pm\xi}(y) )=\xi$ (since $\xi$ is small) and $d^g(R^g_t(y),R^g_{t\pm a}(y))\geq \min\{|a|,\min_{0<k\leq|a|g_{\min}^{-1}+1}\|k\beta\|\} > 50\xi$ by the choice of $\xi$.
The obtained contradiction finishes the proof in the first case.
\medskip

\textbf{Case 2.} \textit{Both $\alpha$ and $\beta$ have bounded type.}
In this case we argue completely analogously to the Case~1 using Propositions \ref{pr3} and \ref{pr4} instead of Propositions \ref{pr1} and \ref{pr2}.
\end{proof}

\section{Concluding remarks}
Two probability preserving flows $\cT=(T_t)_{t\in\R}$ on $(X,\mu)$ and $\cR=(R_t)_{t\in\R}$ on $(Y,\nu)$ are called \emph{weakly isomorphic} if $\cT$ is a factor of $\cR$ and vise versa, i.e.\ if there exist measure preserving maps $S_1\colon X\to Y$ and $S_2\colon Y\to X$ such that $S_1\circ T_t = R_t\circ S_1$ and $S_2\circ R_t = T_t\circ S_2$ for every $t\in\R$.
Since in the proof of Theorem~\ref{main} we did not use an assumption that the conjugacy $S$ is invertible, in fact, a stronger statement follows: if $|A_f|\neq|A_g|$, then the von Neumann flows $\cT=(T^f_t)_{t\in\R}$ and $\cR=(R^g_t)_{t\in\R}$ are not weakly isomorphic.

\section{Acknowledgements}
The authors would like to thank C.~Ulcigrai for drawing our attention to \cite{Ly88}.
The research leading to these results has received funding from the European Research Council under the European Union's Seventh Framework Programme (FP/2007-2013) / ERC Grant Agreement n.~335989.

\bigskip

\textsc{Department of Mathematics, Pennsylvania State University
}\par\nopagebreak
\textit{E-mail address}:\: \texttt{adkanigowski@gmail.com}

\medskip

\textsc{School of Mathematics, University of Bristol
}\par\nopagebreak
\textit{E-mail address}:\: \texttt{solomko.anton@gmail.com}


\begin{thebibliography}{KT97}


\bibitem[FL04]{FL04}
K. Fr\c{a}czek, M. Lema\'{n}czyk,
\textit{A class of special flows over irrational rotations which is disjoint from mixing flows},
Ergod. Th. Dynam. Sys. \textbf{24} (2004), 1083--1095.

\bibitem[FL06]{FL06}
K. Fr\c{a}czek, M. Lema\'{n}czyk,
\textit{On mild mixing of special flows over irrational rotations under piecewise smooth functions},
Ergod. Th. Dynam. Sys. \textbf{26} (2006), 719--738.




\bibitem[KS]{KaSo}
A.~Kanigowski, A.~Solomko,
\textit{On rank of von Neumann special flows}, submitted, arXiv:1608.05596.


\bibitem[KT97]{KT97}
A.~Katok, J.-P.~Thouvenot,
\textit{Slow entropy type invariants and smooth realization
of commuting measure-preserving transformations}, Ann. Inst. H. Poincare, \textbf{33} (1997), 323--338.

\bibitem[Kh35]{Kh}
A.Ya.~Khinchin,
\textit{Continued Fractions},
University of Chicago Press, 1964.

\bibitem[Ko72]{Ko72}
A.V.~Ko\v{c}ergin,
\textit{On the absence of mixing in special flows over the rotation of a circle and in flows on a two-dimensional torus},
Dokl. Akad. Nauk SSSR \textbf{205} (1972), 949--952.

\bibitem[Ly88]{Ly88}
R.~Lyons,
\textit{Strong laws of large numbers for weakly correlated random variables},
Michigan Math.~J. \textbf{35} (1988), 353--359.

\bibitem[Or70]{Orn}
D.~Ornstein,
\textit{Bernoulli shifts with the same entropy are isomorphic},
Advances in Math. \textbf{4}  (1970), 337--352.


\bibitem[vN32]{vN32}
J~von~Neumann,
\textit{Zur Operatorenmethode in der klassischen Mechanik},
Ann. of Math. (2) \textbf{33} (1932), 587--642.

\end{thebibliography}
\end{document}